\documentclass[11pt]{amsart}
\usepackage{amsxtra}
\usepackage{amssymb}
\usepackage[mathscr]{eucal}
\theoremstyle{plain}

\newtheorem{theorem}{Theorem}[subsection]
\newtheorem{lemma}[theorem]{Lemma}
\newtheorem{definition-theorem}[theorem]{Definition-Theorem}
\newtheorem{proposition}[theorem]{Proposition}
\newtheorem{corollary}[theorem]{Corollary}
\newtheorem{definition}[theorem]{Definition}
\newtheorem{example}[theorem]{Example}
\newtheorem{remark}[theorem]{Remark}
\newtheorem{conjecture}[theorem]{Conjecture}
\newtheorem{notation}[theorem]{Notation}

\newtheorem{assumption}[theorem]{Assumption}
\newtheorem*{maintheorem*}{Main Theorem}
\newcommand \bth[1] { \begin{theorem}\label{t#1} }
\newcommand \ble[1] { \begin{lemma}\label{l#1} }

\newcommand \bpr[1] { \begin{proposition}\label{p#1} }
\newcommand \bco[1] { \begin{corollary}\label{c#1} }
\newcommand \bde[1] { \begin{definition}\label{d#1}\rm }
\newcommand \bex[1] { \begin{example}\label{e#1}\rm }
\newcommand \bre[1] { \begin{remark}\label{r#1}\rm }
\newcommand \bcj[1] { \begin{conjecture}\label{j#1}\rm }

\newcommand \bnota[1] { \begin{notation}\label{n#1}\rm }
\renewcommand {\eth} { \end{theorem} }
\newcommand {\ele} { \end{lemma} }

\newcommand {\epr} { \end{proposition} }
\newcommand {\eco} { \end{corollary} }
\newcommand {\ede} { \end{definition} }
\newcommand {\eex} { \end{example} }
\newcommand {\ere} { \end{remark} }
\newcommand {\ecj} { \end{conjecture} }

\newcommand {\enota} { \end{notation} }

\makeatletter
\newsavebox{\@brx}
\newcommand{\llangle}[1][]{\savebox{\@brx}{\(\m@th{#1\langle}\)}%
  \mathopen{\copy\@brx\kern-0.5\wd\@brx\usebox{\@brx}}}
\newcommand{\rrangle}[1][]{\savebox{\@brx}{\(\m@th{#1\rangle}\)}%
  \mathclose{\copy\@brx\kern-0.5\wd\@brx\usebox{\@brx}}}
\makeatother

\usepackage{tikz}
\usetikzlibrary{matrix}
\usepackage{graphicx,ctable,booktabs}
\usetikzlibrary{arrows.meta}
\usepackage{tikz-cd}

\DeclareMathOperator{\Ext}{Ext}

 \DeclareMathOperator{\Proj}{Proj}

\DeclareMathOperator{\Aut}{Aut}

\DeclareMathOperator{\stmod}{{\sf stmod}}
\DeclareMathOperator{\StMod}{{ \sf StMod}}

\DeclareMathOperator{\Spc}{Spc}
\DeclareMathOperator{\CPSpc}{CP-Spc}

\DeclareMathOperator{\ev}{ev}
\DeclareMathOperator{\coev}{coev}

\newcommand{\rootsys}{\mathcal{R}}

\newcommand{\mf}{\mathfrak}
\newcommand{\mc}{\mathcal}

\newcommand{\id}{\operatorname{id}}

\newcommand{\kk}{\Bbbk}

\newcommand{\bK}{\mathbf K}

\newcommand{\bP}{\mathbf P}
\newcommand{\bQ}{\mathbf Q}
\newcommand{\bI}{\mathbf I}
\newcommand{\bJ}{\mathbf J}

\newcounter{listequation}

\numberwithin{equation}{subsection}

\begin{document}

\title[Tensor product property for support maps]{Noncommutative Tensor Triangular Geometry and the tensor product property for support maps}

\author[Daniel K. Nakano]{Daniel K. Nakano}
\address{Department of Mathematics \\
University of Georgia \\
Athens, GA 30602\\
U.S.A.}
\thanks{Research of D.K.N. was supported in part by NSF grants DMS-1701768 and DMS-2101941.}
\email{nakano@math.uga.edu}

\author[Kent B. Vashaw]{Kent B. Vashaw}
\address{
Department of Mathematics \\
Louisiana State University \\
Baton Rouge, LA 70803 \\
U.S.A.}
\thanks{Research of K.B.V. was supported by a Board of Regents LSU fellowship and in part by NSF grant DMS-1901830.}
\email{kvasha1@lsu.edu}
\author[Milen T. Yakimov]{Milen T. Yakimov}
\address{
Department of Mathematics \\
Northeastern University \\
Boston, MA 02115 \\
U.S.A.}
\thanks{Research of M.T.Y. was supported in part by NSF grants DMS-1901830 and DMS-2131243 and a Bulgarian Science Fund grant DN02/05.}
\email{m.yakimov@northeastern.edu}
\begin{abstract} 
The problem of whether the cohomological support map of a finite dimensional 
Hopf algebra has the tensor product property has attracted a lot of attention 
following the earlier developments on representations of finite group schemes. 
Many authors have focussed on concrete situations where positive and negative results 
have been obtained by direct arguments. 

In this paper we demonstrate that it is natural to study 
questions involving the tensor product property in the broader setting of 
a monoidal triangulated category. We give an intrinsic characterization
by proving that the tensor product property for the universal support datum is 
equivalent to complete primeness of the categorical spectrum. From these results one  
obtains information for other support data, including the cohomological one.
Two theorems are proved giving compete primeness and non-complete primeness 
in certain general settings. 

As an illustration of the methods, we give a proof of a recent conjecture of Negron and Pevtsova on the tensor product property for the 
cohomological support maps for the small quantum Borel algebras for all complex simple Lie algebras. 
\end{abstract}
\maketitle
\section{Introduction} \label{Intro}

\subsection{Monoidal Triangular Geometry} Tensor triangular geometry as introduced by Balmer has played a unifying role in understanding the interrelationships between 
representation theory, homological algebra and commutative ring theory/algebraic geometry. In \cite{NVY1}, the authors developed a 
noncommutative version of Balmer's {\em{tensor triangular geometry}} \cite{Balmer1}. Our new theory has the advantage that it 
can be applied to a wider variety of categories such as the stable module category for any finite-dimensional Hopf algebra. 
Given a monoidal triangulated category $\bK$, we associated
\begin{itemize}
\item a topological space $\Spc \bK$ of (thick) prime ideals and
\item a support datum map $V : \bK \to  \mc{X}_{sp}(\Spc \bK)$ to the set of specialization closed subsets of $\Spc \bK$, 
\end{itemize}
and we proved that this support datum is a universal final object in the category of all support data, see Theorem \ref{T:universality} below.

As in the case for non-commutative rings, for monoidal tensor categories, we demonstrated that it was important to distinguish various types of prime ideals. 
The definition of a {\em prime ideal} in this setting involves considering products of ideals whereas the definition of a {\em completely prime ideal} entails 
considering products of objects in the category. The notion of {\em semiprime ideal} is also a key concept in this new theory.

\subsection{Support Theory} The precursor to support data, namely {\em support varieties}, were first developed in the context of modular representations of 
finite groups by the pioneering work of Alperin and Carlson. Since that time, in representation theory (and in the more general setting of monoidal triangulated 
categories) there has been a plethora of contexts where support theory has been studied which includes 
\begin{enumerate}
\item[(i)] the cohomological support via group, Hopf algebra and Hochschild cohomology \cite{Carlson,SS}, 
\item[(ii)] the rank variety and $\Pi$-support via embedded subobjects \cite{Carlson,FP}, 
\item[(iii)] support via actions of commutative algebras \cite{BIK1,BIK2}, 
\item[(iv)] support via actions of the extended endomorphism ring of the identity object \cite{BKSS}, 
\item[(v)] support via tensor triangular geometry \cite{Balmer1}, 
\end{enumerate}
and other approaches. Many fundamental connections between these support theories have been established.

In the aforementioned cases, a {\em{support datum map}} is a map $\sigma$ from the objects of a monoidal triangulated category $\bK$ to the set of specialization closed subsets 
$\mc{X}_{sp}(X)$ of a topological space $X$. The following problem has attracted a lot of attention and has been at the heart of applications of support maps:
\medskip
\\
{\bf{Problem.}} When does a support datum $\sigma : \bK \to \mc{X}_{sp}(X)$ possess the {\em{tensor product property}}
\[
\sigma(A \otimes B) = \sigma (A) \cap \sigma(B), \quad \forall A, B \in \bK?
\] 

For the cohomological support for modular representations of finite groups this was proved in \cite{Carlson} and for finite group schemes in \cite{FP}. 
In the support setting in (iii), a positive answer was obtained in \cite{BIK1,BIK2} under a stratification assumption. There has been a great deal 
of research on this problem for the cohomological support for the stable module category $\StMod(H)$ of a finite dimensional Hopf
algebra $H$. In concrete situations positive and negative answers were obtained in \cite{BW1,FW1,NP,PW}.

\subsection{Main Results} The main goal of this paper is to illustrate how the tensor product property can be characterized in terms of the 
intrinsic structure of the underlying monoidal triangulated category. More specifically, the main results of this paper are as follows:
\begin{enumerate}
\item[(i)] Given a monoidal triangulated category, we prove that the universal support datum $V : \bK \to  \mc{X}_{sp}(\Spc \bK)$ has the tensor product property
if and only if all prime ideals of $\bK$ are completely prime (Theorem \ref{equiv}). 
\item[(ii)] We prove that if all thick right ideals of a monoidal triangulated category $\bK$ are two-sided, then the property in (i) holds for $\bK$ (Theorem \ref{compl-prime}). 
\item[(iii)] We show that if every object of a monoidal triangulated category is either left or right dualizable and the category has a nilpotent object, then 
the property in (i) does not hold for $\bK$ (Theorem \ref{non-compl-prime}). 
\end{enumerate}

The power of Theorem \ref{T:universality} is that the verification of the support property for individual objects of the category $\bK$ is shown to be equivalent to an intrinsic 
global property of the Balmer spectrum $\Spc \bK$ of the category. In noncommutative ring theory, the question of whether all prime deals of a noncommutative 
ring are completely prime is a much studied one. Dixmier proved that the universal enveloping algebra $U({\mathfrak{g}})$ of a finite dimensional Lie algebra 
has this property if and only if the Lie algebra ${\mathfrak{g}}$ is solvable \cite[Theorem 3.7.2]{Dixmier}. For general noncommutative rings there are no classification 
theorems of this sort; positive results for quantum function algebras and Cauchon--Goodearl--Letzter extensions were obtained in \cite{J1,J,GL}. 
Theorem \ref{T:universality} establishes a bridge between the tensor product property for support data and the categorical versions of these questions in ring theory. 

Theorems \ref{compl-prime} and \ref{T:universality} allow for a fast checking of the tensor product property in many interesting situations. This can be combined with 
\cite[Theorems 6.2.1 and 7.3.1]{NVY1} where we proved that support data satisfying natural assumptions coincide with the universal support map $V : \bK \to \mc{X}_{sp}(\Spc \bK)$.
One can use this to apply Theorems \ref{compl-prime} and \ref{T:universality} to verify whether other support maps for a monoidal triangulated category $\bK$, for instance the 
cohomological support map, possess the tensor product property. Along this path we obtain the last main result in the paper, proving the Negron and Pevtsova conjecture \cite{NP} that
\begin{itemize}
\item[(iv)] the cohomological support maps for all small quantum Borel algebras
associated to arbitrary complex simple Lie algebras and arbitrary choices of group-like elements possess the tensor product property. 
\end{itemize}

\subsection{Acknowledgements} We are grateful to the anonymous referee whose valuable comments and suggestions helped us to improve the results in the paper and the exposition.
We thank Cris Negron for useful discussions along with comments on an earlier version of our manuscript.

\section{Preliminaries on noncommutative tensor triangular geometry}

\subsection{Monoidal Triangulated Categories} 

We follow the conventions in \cite{NVY1}. A {\em{monoidal triangulated category}} (M$\Delta$C for short) is a monoidal category $\bK$ in the sense \cite[Definition 2.2.1]{EGNO1} which is triangulated and for which the monoidal structure $\otimes : \bK \times \bK \to \bK$ is an exact bifunctor. 

Recall that a {\em{thick subcategory}} of a triangulated category $\bK$ is a full triangulated subcategory of $\bK$ that contains all direct summands of its objects. 
A {\em{thick right}} (resp. {\em{two-sided}}) {\em{ideal}} of an M$\Delta$C, $\bK$, is a thick subcategory of $\bK$ that is closed under right tensoring (resp. right and left tensoring) with  
arbitrary objects of $\bK$. For each object $M \in \bK$ there exist unique minimal right and two-sided ideals containing $M$, which will be denoted by
$\langle M \rangle_{\mathrm{r}}$ and $\langle M \rangle$, respectively.

\subsection{Prime Ideals and the Balmer Spectrum} We call a proper two-sided ideal  $\bP$ of $\bK$ {\em{prime}} if 
\begin{equation}
\label{inclPIJ}
\bI \otimes \bJ \subseteq \bP \Rightarrow \bI \subseteq \bP \; \; \mbox{or} \; \; \bJ \subseteq \bP
\end{equation}
for all thick two-sided ideals $\bI$ and $\bJ$ of $\bK$. This property is equivalent to saying that \eqref{inclPIJ} holds for all pairs of 
thick right ideals $\bI$ and $\bJ$ of $\bK$. It is also equivalent to the condition that for all $A, B \in \bK$, 
\[
A \otimes C \otimes B \in \bP, \forall C \in \bK \Rightarrow  A \in \bP \; \; \mbox{or} \; \; B \in \bP,
\]
see \cite[Theorem 3.2.2]{NVY1}. 

One can define a notion of primeness on objects of $\bK$ as follows. An ideal $\bP$ is {\em completely prime} if and only if 
\[
A\otimes B\in \bP \Rightarrow A\in \bP \; \; \mbox{or} \; \; B \in \bP
\]
for all objects $A$ and $B$ in $\bK$. 

With these definitions of primeness, one can define a topological space that is analogous to the spectrum of a non-commutative ring. 

\begin{definition}\label{spc}
\begin{itemize}
\item[(a)] The noncommutative Balmer spectrum $\Spc \bK$ of an M$\Delta$C, $\bK$, is the set of its prime ideals with the topology generated by the 
closed sets 
\[
V(M) = \{ \bP \in \Spc \bK \mid M \not \in \bP \}
\]
for $M \in \bK$.
\item[(b)] Let $\CPSpc \bK$ be the topological subspace consisting of all completely prime ideals of $\bK$.
Its topology is generated by the sets 
\[
V_{CP}(M) = \{ \bP \in \CPSpc \bK \mid M \not \in \bP \}
\]
for $M \in \bK$.
\end{itemize} 
\end{definition}

From the definitions, one can easily verify that every completely prime ideal in an M$\Delta$C is prime. Therefore, one has 
$$ 
V_{CP}(M)=V(M)\cap \CPSpc \bK.
$$ 
It is clear that an intersection of prime ideals need not be a prime ideal. 
\begin{definition}
\label{semiprime-def}
A semiprime ideal of an M$\Delta$C, $\bK$, is an intersection of prime ideals of $\bK$.
\end{definition} 
The following characterization of semiprime ideals was proved in \cite[Theorem 3.4.2]{NVY1}:
\begin{theorem} 
\label{semiprime} The following are equivalent for a proper thick ideal $\bQ$ of an M$\Delta$C, $\bK$:
\begin{enumerate}
\item[(a)] $\bQ$ is a semiprime ideal;
\item[(b)] For all $A \in \bK$, if $A \otimes C \otimes A \in \bQ$, $\forall C \in \bK$, then $A \in \bQ;$
\item[(c)] If $\bI$ is any thick two-sided ideal of $\bK$ such that $\bI \otimes \bI \subseteq \bQ$, then $\bI \subseteq \bQ$;
\item[(d)] If $\bI$ is any thick right ideal of $\bK$ such that $\bI \otimes \bI \subseteq \bQ$, then $\bI \subseteq \bQ$.
\end{enumerate}
\end{theorem}

\subsection{Support data maps, universality of $\Spc \bK$} 
\label{support}
One of the important features about monoidal triangulated 
categories is the use of maps that take objects of $\bK$ to subsets of a topological space. 
For a given topological space $Y$, we will denote by $\mc{X}(Y)$, $\mc{X}_{cl}(Y)$ and $\mc{X}_{sp}(Y)$ the collections of its
subsets, closed subsets and specialization closed subsets, respectively.
Given a map $\sigma : \bK \to \mc{X}(Y)$, denote its extension to the set of thick subcategories of $\bK$ given by
\begin{equation}
\label{Phi}
\Phi_\sigma(\bI) = \bigcup_{A \in \bI} \sigma(A).
\end{equation}
\begin{definition}
\label{supp}
A support datum for an M$\Delta$C, $\bK$, is a map
\[
\sigma : \bK \to \mc{X}(Y)
\]
for a topological space $Y$ such that
\smallskip
\begin{enumerate}
\item[(i)] $\sigma(0)=\varnothing $ and $ \sigma(1)= Y$;
\item[(ii)] $\sigma(A\oplus B)=\sigma(A)\cup \sigma(B)$, $\forall A, B \in \bK$; 
\item[(iii)] $\sigma(\Sigma A)=\sigma(A)$, $\forall A \in \bK$; 
\item[(iv)] If $A \to B \to C \to \Sigma A$ is a distinguished triangle, then $\sigma(A) \subseteq \sigma(B) \cup \sigma(C)$:
\item[(v)] $\bigcup_{C \in \bK} \sigma (A   \otimes C \otimes B) = \sigma (A) \cap \sigma(B)$, $\forall A, B \in \bK$. 
\end{enumerate}
\smallskip
A weak support datum is a map
\[
\sigma : \bK \to \mc{X}(Y)
\]
which satisfies conditions (i)-(iv) and the condition 
\smallskip
\begin{enumerate}
\item[(v')] $\Phi_\sigma(\bI \otimes \bJ) = \Phi_\sigma(\bI) \cap \Phi_\sigma(\bJ)$ for all thick two-sided ideals $\bI$ and $\bJ$ of $\bK$.
\end{enumerate}
\smallskip
A quasi support datum is a map
\[
\sigma : \bK \to \mc{X}(Y)
\]
which satisfies conditions (i-iv) and the condition 
\smallskip
\begin{enumerate}
\item[(v'')] $\sigma(A \otimes B)\subseteq \sigma (A)$, for all $A, B \in \bK$. 
\end{enumerate}
\end{definition}
We note that condition (v'') is equivalent to requiring that $\Phi_\sigma ( \langle A \rangle_{\mathrm{r}} ) = \sigma (A)$ for all $A \in \bK$.
For M$\Delta$Cs, $\bK$ admitting arbitrary set indexed coproducts, we will consider maps $\sigma : \bK \to \mc{X}(Y)$ satisfying 
the stronger property
\smallskip
\begin{enumerate}
\item[(ii')] $\sigma(\bigoplus_{i \in I} A_i)=\bigcup_{i \in I} \sigma(A_i )$, $\forall A_i \in \bK$ and all sets $I$
\end{enumerate}
\smallskip
in place of property (ii); this property will be explicitly mentioned when used.

Each support datum is a weak support datum \cite[Lemma 4.3.1 and 4.5.1]{NVY1}. For every  M$\Delta$C $\bK$, the map 
\[
V : \bK \to \mc{X}_{cl} (\Spc \bK) \quad 
\mbox{given by} \quad
V(A) = \{ \bP  \in  \Spc \bK : M \notin \bP\}
\]
is a support datum. It is universal as proved in \cite[Theorems 4.2.2 and 4.5.1]{NVY1}:
\begin{theorem} \label{T:universality}
Let $\bK$ be an M$\Delta$C.
\begin{enumerate}
\item[(a)] {\em{The support $V$ is the final object in the collection of support data $\sigma$ for $\bK$ such that $\sigma(A)$ is closed for each $A \in \bK$: 
for any such $\sigma : \bK \to \mc{X}(Y)$, there is a unique continuous map $f_\sigma: Y \to \Spc \bK$ satisfying 
\[
\sigma(A)= f_\sigma^{-1}( V(A)) \quad
\mbox{for} 
\quad A \in \bK.
\]
}}
\item[(b)] {\em{The support $V$ is the final object in the collection of weak support data $\sigma$ for $\bK$ such that $\Phi_{\sigma}(\langle A \rangle)$ is closed 
for each $A \in \bK$: for any such $\sigma : \bK \to \mc{X}(Y)$, there is a unique continuous map $f_\sigma: Y \to \Spc \bK$ satisfying 
\[
\Phi_{\sigma}(\langle A \rangle)= f_\sigma^{-1}(V(A)) \quad
\mbox{for} 
\quad A \in \bK.
\]
}}
\end{enumerate}
\end{theorem}

\label{Background}
\section{The tensor product property for the universal support datum of a monoidal triangulated category} 

\subsection{Complete Primeness of $\Spc$ and the Tensor Product Property} We begin by proving a theorem 
that indicates how the structural properties of a monoidal triangulated category are captured by characterizations 
involving the universal support datum. 

\begin{theorem}
\label{equiv} 
For every monoidal triangulated category $\bK$, the following are equivalent: 
\begin{enumerate}
\item[(a)] The universal support datum $V : \bK \to \mc{X} (\Spc \bK)$ has the tensor product property 
\[
V( A \otimes B) = V(A) \cap V(B), \quad \forall A, B \in \bK. 
\] 
\item[(b)] Every prime ideal of $\bK$ is completely prime.
\end{enumerate} 
\end{theorem}
\begin{proof}
(a $\Rightarrow$ b) Let $\bP \in \Spc \bK$ and $A, B \in \bK$ be such that $A \otimes B \in \bP$. Then 
\[
\bP \notin V(A \otimes B) = V(A) \cap V(B).
\]
Hence, either $\bP \notin V(A)$ or $\bP \notin V(B)$, and thus, either $A \in \bP$ or $B \in \bP$. 

(b $\Rightarrow$ a) For $A, B \in \bK$, we have
\begin{align*} 
\Spc \bK \backslash V(A \otimes B) &= \{ \bP \in \Spc \bK \mid A \otimes B \in \bP \} \\
&= \{ \bP \in \Spc \bK \mid A \in \bP \} \cup   \{ \bP \in \Spc \bK \mid B \in \bP \} \\
&= (\Spc \bK \backslash V(A)) \cup (\Spc \bK \backslash V(B)).  
\end{align*}
Thus $V(A \otimes B)  = V(A) \cap V(B)$. 
\end{proof}

The proof of Theorem \ref{equiv} immediately gives the following fact.
\begin{corollary}
\label{cor-compl-prime} For every monoidal triangulated category, $\bK$, the map
\[
V_{\mathrm{CP}} : \bK \to \CPSpc \bK \quad 
\mbox{given by} \quad
V_{\mathrm{CP}} (A) = V(A) \cap \CPSpc \bK
\]
has the tensor product property.
\end{corollary}
In many cases for monoidal triangulated categories, $\bK$, the space $\CPSpc \bK$ can be much smaller than $\Spc \bK$. So in general, the support datum $V_{\mathrm{CP}}$ captures much 
less information than the universal support datum $V$.

\subsection{A Criterion for Complete Primeness of $\Spc \bK$} In this section we investigate monoidal tensor categories where the right ideals 
coincide with the two-sided ideals. In this situation, every prime ideal is completely prime and the tensor product property holds. This key observation will 
be applied in Section~\ref{q-Borels}. 

\begin{theorem}
\label{compl-prime} Let $\bK$ be a monoidal triangulated category in which every thick right ideal is two-sided. Then every prime ideal of $\bK$ is completely prime, 
and as a consequence, the universal support datum $V : \bK \to \mc{X} (\Spc \bK)$ has the tensor product property 
\[
V( A \otimes B) = V(A) \cap V(B), \quad \forall A, B \in \bK. 
\] 
\end{theorem}
\begin{proof} First we claim that
\begin{equation}
\label{eq:ideals}
\langle M \rangle_{\mathrm{r}} = \langle M \rangle, \quad \forall M \in \bK.
\end{equation}
The inclusion $\langle M \rangle_{\mathrm{r}} \subseteq \langle M \rangle$ is obvious. The reverse inclusion is proved as follows. 
The hypothesis states that $\langle M \rangle_{\mathrm{r}}$ is a a two-sided thick ideal and, in particular, it contains
$\langle N \rangle$ for all $N \in \langle M \rangle_{\mathrm{r}}$. Applying this for $N=M$ yields 
$\langle M \rangle_{\mathrm{r}} \supseteq \langle M \rangle$.

Let $\bP \in \Spc \bK$ and $A, B \in \bK$ be such that $A \otimes B \in \bP$. Therefore $A \otimes \langle B \rangle_{\mathrm{r}} \subseteq \bP$ and, by 
\eqref{eq:ideals}, $A \otimes \langle B \rangle \subseteq \bP$. This implies that $A \otimes C \otimes B \in \bP$ for all $C \in \bK$ and, by the primeness 
of $\bP$, $A \in \bP$ or $B \in \bP$. Therefore, the thick ideal $\bP$ is completely prime. The second statement follows from the first and Theorem \ref{equiv}.
\end{proof}
If a monoidal triangulated category $\bK$ has the property that $A \otimes B \cong B \otimes A$ for all $A, B \in \bK$, then $\bK$ satisfies the assumption of Theorem \ref{compl-prime}.
This in particular holds for all braided monoidal triangulated categories. The next section contains much more nontrivial applications of this theorem.

\section{A criterion for non-complete primeness of $\Spc \bK$} 

\subsection{Rigidity and Semi-Primeness} Recall that an object $A$ of a monoidal category $\bK$ is {\em{left dualizable}} if there exists an object $A^*$ (called the {\em{left dual of $A$}}), 
together with evaluation and coevaluation maps
\[
\ev: A^* \otimes A \to 1 \quad \mbox{and} \quad
\coev: 1 \to A \otimes A^*,
\] 
such that the compositions
\begin{equation}
\label{eq:rigid}
A \xrightarrow{\coev \otimes \id} A \otimes A^* \otimes A \xrightarrow{ \id \otimes \ev} A \quad \mbox{and} \quad
A^* \xrightarrow{\id \otimes \coev} A^* \otimes A \otimes A^* \xrightarrow{\ev \otimes \id} A^*
\end{equation}
are the identity maps on $A$ and $A^*$, respectively. The left dual object $A^*$ is unique up to a unique isomorphism, 
\cite[Proposition 2.10.5]{EGNO1}. In a similar way one defines the notions of {\em{right dualizable}} objects and their {\em{right duals}}, 
see \cite[Definition 2.10.2]{EGNO1}. Finally, an object of a monoidal category is rigid if it is both left and right dualizable.
\begin{proposition}
\label{rigid-semiprime} If $\bK$ is a monoidal triangulated category in which every object is either left or right dualizable, then every thick ideal of $\bK$ is semiprime.
\end{proposition}
\begin{proof}
Fix a thick two-sided ideal $\bI$ of $\bK$. Let $A \in \bK$ be such that $A \otimes B \otimes A \in \bI$ for all $B \in \bK$. In particular, 
$A \otimes A^* \otimes A \in \bI$. Assume that $A$ is left dualizable; the case when it is right dualizable is handled in a similar fashion.
It follows from \eqref{eq:rigid} that $A$ is a direct summand of $A \otimes A^* \otimes A$. Since $\bI$ is a thick subcategory 
of $\bK$, $A \in \bI$. Theorem \ref{semiprime} now implies that $\bI$ is a semiprime ideal of $\bK$. 
\end{proof}

\subsection{Existence of Nilpotent Elements} Given a monoidal tensor category where every object is either left or right dualizable, one can now show that the existence of a 
nilpotent element insures that the universal support datum does not satisfy the tensor product property. 

\begin{theorem}
\label{non-compl-prime} Let $\bK$ be a monoidal triangulated category in which every object is either left or right dualizable. If $\bK$ has a non-zero nilpotent object $M$ (i.e., $M \not\cong 0$ but 
$M^{\otimes n}:=M \otimes\dots \otimes M \cong 0$, for some $n>0$) then not all prime ideals of $\bK$ are completely prime. As a consequence, the universal support datum $V : \bK \to \mc{X} (\Spc \bK)$ does not have the tensor product property. 
\end{theorem} 
\begin{proof} By Proposition \ref{rigid-semiprime}, $\langle 0 \rangle$ is a semiprime ideal of $\bK$. Hence, the prime radical of $\bK$ equals $\langle 0 \rangle$. 

On the other hand $M$ lies in all completely prime ideals $\bP$ of $\bK$ because $M^{\otimes n} \cong 0 \in \bP$. If all prime ideals of $\bK$ are 
completely prime, this would imply that $M$ belongs to the prime radical of $\bK$ (i.e., $M \in \langle 0 \rangle$), which is a contradiction.
\end{proof}

The following corollary follows from Theorem \ref{non-compl-prime}, because all objects of $\stmod(H)$ are rigid for finite dimensional Hopf algebras $H$. 

\begin{corollary}
\label{Hopf-non-compl-prime}
Assume that $H$ is a finite dimensional Hopf algebra which admits a non-projective finite dimensional module $M$ such that $M^{\otimes n}$ is projective. 
Then not all prime ideals of the stable module category  $\stmod(H)$ are completely prime, i.e., 
the universal support datum $V : \bK \to \mc{X} (\Spc (\stmod(H)))$ does not have the tensor product property. 
\end{corollary}  

The following corollary of Theorem \ref{non-compl-prime} is of independent interest.

\begin{corollary}
\label{A-B}
If $\bK$ is a monoidal triangulated category in which every object is either left or right dualizable and $\bK$ has objects $A$ and $B$, 
such that $A \otimes B \cong 0$ but $B \otimes A \not\cong 0$, then not all prime ideals of $\bK$ are completely prime, i.e., 
the universal support datum $V : \bK \to \mc{X} (\Spc \bK)$ does not have the tensor product property. 
\end{corollary}

This follows from Theorem \ref{non-compl-prime}, because $M:=B \otimes A$ is not the zero object in $\bK$, but 
$M \otimes M \cong B \otimes (A \otimes B) \otimes A \cong 0$. 

\subsection{Remarks on the Work of  Benson-Witherspoon} \label{BW1} In \cite{BW1} 
Benson and Witherspoon considered the stable module categories of Hopf algebras of the form 
\[
H_{G,L} := (\kk[G] \# \kk L)^*,
\]
where $G$ and $L$ are finite groups with $L$ acting on $G$ by group automorphisms, $\kk$ is a field of positive characteristic dividing the order of $G$,  
$\kk L$ is the group algebra of $L$, $\kk[G]$ is the dual of the group algebra of $G$, and $\#$ denotes the corresponding smash product. 

Let $p$ be a prime number and $n$ be a positive integer. In \cite[Example 3.3]{BW1} Benson and Witherspoon proved that
for $G := ({\mathbb Z}/ p {\mathbb Z})^n$, $L := {\mathbb Z} /n {\mathbb Z}$ (with $L$ cyclically permuting the factors of $G$) and $\kk$ a field of characteristic 
$p$, $H_{G,L}$ admits a non-projective finite dimensional module $M$ such that $M \otimes M$ is projective. By Corollary~\ref{Hopf-non-compl-prime}, the universal 
support datum map for $\stmod(H_{G,L})$ does not satisfy the tensor product property. 

Benson and Witherspoon constructed \cite[Example 3.2]{BW1} a Hopf algebra of the form $H_{G,L}$ such that 
$H_{G,L}$ has a pair of finite dimensional representations $A, B$ with the property that $A \otimes B$ is not projective, 
but $B \otimes A$ is is projective. The group $G$
is chosen to be the Klein 4-group, $L$ is the cyclic group of order 3 whose generator 
cyclically permutes the non-identity elements of $G$, and the field $\kk$ has characteristic 2.
By Corollary \ref{A-B}, for this Hopf algebra $H_{G,L}$, the universal 
support datum map for $\stmod(H_{G,L})$ does not satisfy the tensor product property either.

\bre{mistake-BKS} We note that \cite[Lemma 7.10]{BKS} implies that every prime ideal of every monoidal triangulated category is completely prime.  
A counterexample to the statement is provided in the aforementioned example. The gap is in the converse direction in the proof of \cite[Lemma 7.10]{BKS} where it states that the 
converse direction is analogous. 

\ere

\section{The tensor product property for the cohomological support for small quantum Borels}

\subsection{Preliminaries} 
\label{sec:prelim-qg}
Let $\rootsys$ be an irreducible root system of rank $n$. Let $\ell$ be a positive integer and 
$\zeta$ be a primitive $\ell$th root of unity. 

We begin by introducing a general construction of the small quantum group 
for a Borel algebra that generalizes the well-known construction using group like elements arising from the root lattice. 
All of these will be finite dimensional Hopf algebras.
For a given $\rootsys$, let $X$ be the corresponding weight lattice and $\rootsys^{+}$ be a set of positive roots.
Denote by $\{\alpha_{1},\dots,\alpha_{n}\}$ the base of simple roots for $\rootsys$ corresponding to $\rootsys^+$
and by $\{d_1, \ldots, d_n\}$ the collection of relatively prime positive integers that symmetrizes 
the corresponding Cartan matrix. Denote by $\langle -, - \rangle$ the Weyl group invariant nondegenerate symmetric inner product on the Euclidean space 
${\mathfrak t}^*_{\mathbb R}$ spanned by $\rootsys$, normalized by $\langle \beta, \beta \rangle =2$ for short roots $\beta$. In terms of this form, 
the integers $d_i$ are given by $d_i = \langle \alpha_i , \alpha_i \rangle/ 2$.
Let  $\{\alpha_{1}\spcheck,\dots,\alpha_{n}\spcheck\}$ be the corresponding coroots thought of as elements of ${\mathfrak t}^*_{\mathbb R}$ by setting 
\[
\alpha_i\spcheck = \frac{2 \alpha_i}{ \langle \alpha_i, \alpha_i \rangle} = \frac{\alpha_i}{d_i} \cdot
\]

Choose a ${\mathbb Z}$-lattice, $\Gamma$, with ${\mathbb Z}\rootsys \subseteq \Gamma \subseteq X$. Such a lattice $\Gamma$ has rank $n$. Let 
$\{\mu_{1},\dots,\mu_{n}\}$ be a ${\mathbb Z}$-basis for $\Gamma$. 

Let $u_{\zeta}({\mathfrak b})$ be the small quantum group as described in \cite[Section 2.2]{BNPP}. Then 
$u_{\zeta}({\mathfrak b})=u_{\zeta}({\mathfrak u}) \# u_{\zeta}({\mathfrak t})$ where 
$u_{\zeta}({\mathfrak u})$ is generated by the root vectors $\{E_{\beta} \mid \beta\in \rootsys^{+}\}$ satisfying $E_\beta^\ell=0$ and
$u_{\zeta}({\mathfrak t})$ is a Hopf algebra isomorphic to the group algebra of ${\mathbb Z} \rootsys /( \ell {\mathbb Z} \rootsys)$ over ${\mathbb C}$,
realized as
\[
u_{\zeta}({\mathfrak t}) = {\mathbb C} [K_{\alpha_{1}}^{\pm 1},\dots,K_{\alpha_{n}}^{\pm 1}]/(K_{\alpha_i}^\ell-1, 1 \leq i \leq n)
\]
where $K_{\alpha_{i}}$ are group like elements.  The relations in $u_{\zeta}({\mathfrak b})$ defining the smash product are
\begin{equation} 
K_{\alpha_{i}} E_{\beta} K_{\alpha_{i}}^{-1}=\zeta^{\langle \beta,\alpha_{i} \rangle} E_{\beta}
\end{equation} 
for $\beta\in \rootsys^{+}$. 

We can consider the following generalization of the small quantum group for the Borel subalgebra. 
Given a lattice $\Gamma$ with ${\mathbb Z}\rootsys \subseteq \Gamma \subseteq X$ as above, define its sublattice
\[
\Gamma':= \{ \nu \in \Gamma \mid \langle \nu, \rootsys \rangle \subseteq \ell {\mathbb Z} \}. 
\]
Obviously, $\Gamma' \supseteq \ell \Gamma$, so $\Gamma/\Gamma'$ is a factor group of $\Gamma / \ell \Gamma \cong ( {\mathbb Z} / \ell  {\mathbb Z} )^n$. 
Denote the canonical projection
\begin{equation}
\label{bar}
\Gamma \twoheadrightarrow \Gamma/ \Gamma' \quad \mbox{by} \quad \mu \mapsto \overline{\mu}.
\end{equation}
Let
\begin{equation}
\label{group-alg}
u_{\zeta,\Gamma}({\mathfrak t}) \; \; \mbox{denote the group algebra of $\Gamma/\Gamma'$ over ${\mathbb C}$}. 
\end{equation}
For $\mu \in \Gamma/\Gamma'$ denote by $K_\mu$ the element of $u_{\zeta,\Gamma}({\mathfrak t})$ corresponding 
to $\mu$. Consider the Hopf algebra
\[
u_{\zeta,\Gamma}({\mathfrak b})= u_{\zeta}({\mathfrak u}) \# u_{\zeta,\Gamma}({\mathfrak t})
\]
with relations 
\begin{equation}
\label{u-act}
K_{\mu} E_{\alpha} K_{\mu}^{-1}=\zeta^{\langle \alpha,\mu_0 \rangle} E_{\alpha} \quad \mbox{for}
\quad \mu \in \Gamma/ \Gamma', \alpha \in \rootsys^+,
\end{equation}
where $\mu_0 \in \Gamma$ is a preimage of $\mu$. By the definition of the lattice $\Gamma'$, the right hand side does not  
depend on the choice of preimage.
The coproduct of the generators $E_{\alpha_i}$ is given by
\begin{equation}
\label{coprod}
\Delta(E_{\alpha_i}) = E_{\alpha_i} \otimes 1 + K_{\overline{\alpha_i}} \otimes E_{\alpha_i}
\end{equation}
for $1 \leq i \leq n$. The antipode is given by $S( E_{\alpha_i}) = - K_{\overline{\alpha_i}}^{-1} E_{\alpha_i}$.

In all of the above definitions, the lattice $\Gamma'$ can be replaced with any sublattice of $\Gamma'$.
The motivation for the use of the full lattice $\Gamma'$ is that this makes $u_{\zeta,\Gamma}({\mathfrak b})$ small 
in the sense that the only group-like central elements of $u_{\zeta,\Gamma}({\mathfrak b})$ are the scalars.
\bre{hopf-hom-two-lat}
Consider two lattices $\Gamma_1$ and $\Gamma_2$ such that ${\mathbb Z}\rootsys \subseteq \Gamma_1 \subseteq \Gamma_2 \subseteq X$.
Then $\Gamma'_1 = \Gamma_1 \cap \Gamma'_2$. Hence, we have a Hopf algebra embedding 
\[
u_{\zeta,\Gamma_1}({\mathfrak b}) \hookrightarrow u_{\zeta,\Gamma_2}({\mathfrak b}) 
\quad
\mbox{given by} \quad K_{\mu + \Gamma'_1} \mapsto K_{\mu + \Gamma'_2},\ E_\alpha \mapsto E_\alpha
\]
for $\mu \in \Gamma_1$, $\alpha \in \rootsys^+$. 
\ere
\subsection{Assumptions on ${\ell}$} For the remainder of this section we will employ one of the 
following assumptions in the statements of our results where $\zeta$ is an ${\ell}$th root of unity. 

\begin{assumption}\label{A:assumption1}  Let ${\ell}$ be a positive integer such that 
\begin{itemize} 
\item[(a)] ${\ell}$ is odd;
\item[(b)] If $\rootsys$ is of type $G_{2}$ then $3\nmid {\ell}$; 
\item[(c)] If $\rootsys$ is of type $A_{1}$ then ${\ell}\geq 3$, otherwise ${\ell}>3$. 
\end{itemize} 
\end{assumption} 
Conditions (a)-(b) in Assumption \ref{A:assumption1} are equivalent to saying that $\ell$ is an odd positive integer which 
is coprime to $\{d_1, \ldots, d_n\}$. 

\begin{assumption}\label{A:assumption2} Let ${\ell}$ be a positive integer such that 
\begin{itemize} 
\item[(a)] ${\ell}$ is odd;
\item[(b)] If $\rootsys$ is of type $G_{2}$ then $3\nmid {\ell}$; 
\item[(c)] ${\ell}>h$ where $h$ is the Coxeter number for $\rootsys$. 
\end{itemize} 
\end{assumption} 
\noindent 
Note that if ${\ell}$ satisfies Assumption~\ref{A:assumption2} then ${\ell}$ satisfies Assumption~\ref{A:assumption1}. 

The group of group-like elements of $u_{\zeta,\Gamma}({\mathfrak t})$ is isomorphic to $\Gamma/\Gamma'$. Next we explicitly describe this 
finite abelian group.

\begin{proposition}  \label{P:G'}
\begin{enumerate}
\item[(a)] If $\ell$ is coprime to $\{d_1, \ldots, d_n\}$, then 
\[
\Gamma' = \Gamma \cap \ell X.
\]
That is, $\Gamma/ \Gamma' \cong \Gamma/ ( \Gamma \cap \ell X)$.
\item[(b)] If $\ell$ is coprime to $\{d_1, \ldots, d_n\}$ and $|X/\Gamma|$, then 
\[
\Gamma' = \ell \Gamma.
\]
That is, $\Gamma/ \Gamma' \cong \Gamma/ (\ell \Gamma) \cong  ( {\mathbb Z} / \ell  {\mathbb Z} )^n$.
\end{enumerate}
\end{proposition}
\begin{proof} (a) Let $\nu = \sum m_i \omega_i \in \Gamma \subseteq X$ for some $m_i \in {\mathbb Z}$. Then $\nu \in \Gamma' \Leftrightarrow$
\begin{align*}
&\langle \nu, \alpha_i \rangle \in \ell  {\mathbb Z}, \quad\forall 1 \leq i \leq n  \Leftrightarrow
\\
&m_i d_i \in \ell  {\mathbb Z}, \quad \; \;  \forall   1 \leq i \leq n  \Leftrightarrow
\\
&m_i \in \ell  {\mathbb Z}, \quad \; \; \; \; \;  \forall 1 \leq i \leq n  \Leftrightarrow
\\
&\nu \in \Gamma \cap \ell X.
\end{align*}

(b) In view of part (a), we have to prove that under the assumptions in part (b), $\Gamma \cap \ell X = \ell \Gamma$. Clearly, 
\[
\Gamma \cap \ell X \supseteq \ell \Gamma.
\]
For the opposite inclusion, take $\nu \in \Gamma \cap \ell X$. Then the order of $\nu/\ell + \Gamma$ in $X/\Gamma$ divides $\ell$. 
Since $\ell$ is coprime to the order of the group $X/\Gamma$, the order of  $\nu/\ell + \Gamma$ equals 1. Therefore, $\nu/\ell \in \Gamma$, and thus, 
$\nu \in \ell \Gamma$. Hence, $\Gamma \cap \ell X = \ell \Gamma$. 
\end{proof}
\bex{standard-recover} 
\label{ex-X'} The standard notion of a small quantum Borel subalgebra $u_{\zeta}({\mathfrak b})$ is recovered from the above one as follows.
Proposition \ref{P:G'}(b), applied for the root lattice $\Gamma = {\mathbb Z} \rootsys$, implies that, 
if $\ell$ is coprime to $\{d_1, \ldots, d_n\}$ and $|X/{\mathbb Z} \rootsys|$, then 
\[
u_{\zeta, {\mathbb Z}  \rootsys}({\mathfrak b}) \cong u_{\zeta}({\mathfrak b}).
\]
Note that both aforementioned algebras are defined for general values of ${\ell}$, but become isomorphic under the coprimeness conditions. 
\eex
\subsection{Automorphisms, Representations and Cohomology} \label{S:autorepi} In this section we will generalize many of the properties presented in \cite[Section 8.3]{NVY1} for 
$u_{\zeta}({\mathfrak b})$ to $u_{\zeta,\Gamma}({\mathfrak b})$. For the reader's convenience, we will use the same notational conventions. 

Denote the character group of $\Gamma/\Gamma'$ by 
\[
\widehat{\Gamma/\Gamma'}.
\]
By abuse of notation, for $\lambda \in \widehat{\Gamma/\Gamma'}$ we denote by the same symbol the one dimensional representation of $u_{\zeta,\Gamma}({\mathfrak b})$
given by
\[
K_\mu \mapsto \lambda(\mu), \quad
E_\alpha \mapsto 0, \quad \forall \mu \in \Gamma/\Gamma', \alpha \in \rootsys^+.
\]

For each $\lambda \in \widehat{\Gamma/\Gamma'}$, one can define an automorphism, $\gamma_\lambda$ of $u_{\zeta,\Gamma}({\mathfrak b})$ as follows: 
\[
\gamma_\lambda (E_\alpha)=  \lambda(\overline{\alpha}) E_\alpha, \quad \gamma_\lambda(K_\mu)=K_\mu, 
\quad \forall \mu \in \Gamma/\Gamma', \alpha \in \rootsys^+.
\]
Denote the subgroup $\Pi= \{ \gamma_{\lambda} :\ \lambda \in  \widehat{\Gamma/\Gamma'} \} \subseteq \Aut(u_{\zeta,\Gamma}({\mathfrak b}))$. 
For any $u_{\zeta,\Gamma}({\mathfrak b})$-module, $Q$, the automorphism $\gamma_{\lambda}$ can be used to define 
a new module structure on it called the twist: $Q^{\gamma_\lambda}$. The underlying vector space of $Q^{\gamma_\lambda}$ is still $Q$ with the action given by 
$x.m=\gamma_\lambda(x) m$ for all $x \in u_{\zeta,\Gamma}({\mathfrak b})$ and $m \in Q^{\gamma_\lambda}$.  

Let $R=\text{H}^{\bullet}(u_{\zeta,\Gamma}({\mathfrak b}), \mathbb{C})$ be the cohomology ring of $u_{\zeta,\Gamma}({\mathfrak b})$. An automorphism in $\Pi$ acts on the cohomology ring by taking an $n$-fold extension of ${\mathbb C}$ with 
${\mathbb C}$ and twisting each module in the $n$-fold extension to produce a new $n$-fold extension. This provides an action of the group $\Pi$ on the ring $R$. 
The following proposition summarizes properties of the automorphisms in $\Pi$ and how they interact with representations and the cohomology. 

\begin{proposition}  \label{P:properties}
Let $u_{\zeta,\Gamma}({\mathfrak b})$ be the small quantum group for the Borel subalgebra and 
$R=\operatorname{H}^{\bullet}(u_{\zeta,\Gamma}({\mathfrak b}),{\mathbb C})$ be the cohomology ring. 
\begin{itemize}
\item[(a)] The irreducible representations for $u_{\zeta,\Gamma}({\mathfrak b})$ are one-dimensional and are precisely the 
representations $\lambda$ for $\lambda \in \widehat{\Gamma/\Gamma'}$.
\item[(b)] For any $u_{\zeta,\Gamma}({\mathfrak b})$-module, $Q$, and $\lambda \in \widehat{\Gamma/\Gamma'}$ one has 
\begin{equation*} \label{e:twistiso} 
\lambda \otimes Q \otimes \lambda^{-1} \cong Q^{\gamma_{\lambda}}.
\end{equation*} 
\item[(c)] The action of $\Pi$ on $R$ is trivial. 
\item[(d)] The action of $\Pi$ on $\operatorname{Proj}(R)$ is trivial. 
\end{itemize} 
\end{proposition} 

\begin{proof} (a) The relations $E_{\alpha}^{\ell}=0$ for $\alpha \in \rootsys^+$ imply that all root vectors $E_\beta$ are in the radical of the finite dimensional 
algebra $u_{\zeta,\Gamma}({\mathfrak b})$ and so they act by $0$ on every irreducible representation of $u_{\zeta,\Gamma}({\mathfrak b})$. Hence, 
every irreducible representations of $u_{\zeta,\Gamma}({\mathfrak b})$ is an irreducible representation of $u_{\zeta,\Gamma}({\mathfrak t})$, 
which is the group algebra of $\Gamma/\Gamma'$, so the  irreducible representation of $u_{\zeta,\Gamma}({\mathfrak t})$  are precisely 
the representations $\lambda$ for $\lambda \in \widehat{\Gamma/\Gamma'}$.

(b) The isomorphism follows from the coproduct formula \eqref{coprod} and the fact that 
the set $\{K_\mu, E_{\alpha_i} \mid \mu \in \Gamma, i=1, \ldots, n\}$ generates the algebra $u_{\zeta,\Gamma}({\mathfrak b})$.

(c and d)  Note that (d) follows immediately from (c). So to finish the proof we show that the action of $\Pi$ on the cohomology ring $R$ is trivial. 

By using the Lyndon-Hochschild-Serre (LHS) spectral sequence and the fact that the representations for $u_{\zeta,\Gamma}({\mathfrak t})$ are completely reducible
(because $u_{\zeta,\Gamma}({\mathfrak t})$ is isomorphic to the group algebra over ${\mathbb C}$ of a finite group), it follows that 
$R=\operatorname{H}^{\bullet}(u_{\zeta}({\mathfrak u}),{\mathbb C})^{u_{\zeta,\Gamma}({\mathfrak t})}$ with respect to the action \eqref{u-act}
(cf. \cite[Theorem 2.5]{GK}). Consequently, for every weight $\nu \in {\mathbb Z} \rootsys$ of $R$ 
\[
\langle \nu, \Gamma \rangle \subseteq \ell {\mathbb Z} \Rightarrow 
\langle \nu, \rootsys \rangle \subseteq \ell {\mathbb Z} \Rightarrow \nu \in {\mathbb Z} \rootsys \cap \Gamma' \Rightarrow \overline{\nu}=0.
\]
Let $f \in R$ be of weight $\nu$. 
The automorphism $\gamma_{\lambda} \in \Pi$ acts on $f$ by 
\[
\gamma_{\lambda}(f)= \lambda(\overline{\nu}) f = f,
\]
which proves the triviality of the $\Pi$-action on $R$.
\end{proof} 
\subsection{Finite Generation} In order to verify the finite generation conditions on the cohomology, we state the following 
result from \cite[Proposition 5.6.3]{BNPP} on the cohomology for $u_{\zeta}({\mathfrak u})$. 

\begin{theorem}\label{T:ucohofg} Let ${\ell}$ satisfy Assumption~\ref{A:assumption1}, and $\zeta$ be an ${\ell}$th root of unity. There exists a polynomial ring 
$S^{\bullet}({\mathfrak u}^{*})$ such that the following holds: 
\begin{itemize} 
\item[(a)] $\operatorname{H}^{\bullet}(u_{\zeta}({\mathfrak u}),{\mathbb C})$ is finitely generated over $S^{\bullet}({\mathfrak u}^{*})$;
\item[(b)] $\operatorname{H}^{\bullet}(u_{\zeta}({\mathfrak u}),{\mathbb C})$ is a finitely generated ${\mathbb C}$-algebra.
\end{itemize} 
\end{theorem} 

Theorem~\ref{T:ucohofg} allows us to consider the issue of finite generation of cohomology for $u_{\zeta,\Gamma}({\mathfrak b})$. 
The filtration in \cite[Section 2.9]{BNPP} on $u_{\zeta}({\mathfrak u})$ that induces the grading  as in \cite[Lemma 5.6.1]{BNPP} is stable under the action of $K_{\mu_{i}}$, $i=1,2\,\dots,n$. 
Consequently, there exists a spectral sequence 
\begin{equation}\label{eq:spectral1}
E_{1}^{i,j}=\operatorname{H}^{i+j}(\text{gr }u_{\zeta}({\mathfrak u}),{\mathbb C})_{(i)}\Rightarrow \operatorname{H}^{i+j}(u_{\zeta}({\mathfrak u}),{\mathbb C}) 
\end{equation} 
such that 
$$\operatorname{H}^{n}(\text{gr }u_{\zeta}({\mathfrak u}),{\mathbb C})\cong \bigoplus_{2a+b=n} S^{a}({\mathfrak u}^{*})^{[1]}\otimes \Lambda^{b}_{\zeta}.$$ 
Here $S^{\bullet}({\mathfrak u}^{*})^{[1]}$ is the symmetric algebra on ${\mathfrak u}^{*}$ (the dual of ${\mathfrak u}$ and the $[1]$ indicates that $u_{\zeta,\Gamma}({\mathfrak t})$ acts trivially) and 
$\Lambda^{b}_{\zeta}$ is a deformation of the 
exterior algebra on ${\mathfrak u}^{*}$ with generators and relations defined in \cite[Section 2.9]{BNPP}. In the proof of Theorem~\ref{T:ucohofg} (given in \cite[Proposition 5.6.3]{BNPP}), 
it is shown that under the assumptions on ${\ell}$, $d_{r}(S^{\bullet}({\mathfrak u}^{*})^{[1]})=0$ for $r\geq 1$ where $d_{r}$ is the differential on the $E_{r}$-page of the spectral 
sequence (\ref{eq:spectral1}). One can then conclude part (a) of Theorem~\ref{T:ucohofg}. 

Since $u_{\zeta}({\mathfrak u})$ is normal in $u_{\zeta,\Gamma}({\mathfrak b})$ (cf. \cite[Section 2.8]{BNPP}) with quotient 
$u_{\zeta,\Gamma}({\mathfrak t})$, and the filtration is stable under $u_{\zeta,\Gamma}({\mathfrak t})$, it follows that 
$u_{\zeta,\Gamma}({\mathfrak t})$ acts on the spectral sequence (\ref{eq:spectral1}). Furthermore, one can verify that $u_{\zeta,\Gamma}({\mathfrak t})$ acts trivially on $S^{\bullet}({\mathfrak u}^{*})^{[1]}$. 

Since finite-dimensional representations for $u_{\zeta,\Gamma}({\mathfrak t})$ are completely reducible, 
the fixed point functor $(-)^{u_{\zeta,\Gamma}({\mathfrak t})}$ is exact. By using the LHS spectral sequence and the exactness, one shows that 
$$\operatorname{H}^{\bullet}(u_{\zeta,\Gamma}({\mathfrak b}),{\mathbb C})\cong \operatorname{H}^{\bullet}(u_{\zeta}({\mathfrak u}),{\mathbb C})^{u_{\zeta,\Gamma}({\mathfrak t})}.$$ 
Moreover, the fixed point functor can be applied to get a spectral sequence: 
\begin{equation}\label{eq:spectral2}
E_{1}^{i,j}=[\operatorname{H}^{i+j}(\text{gr }u_{\zeta}({\mathfrak u}),{\mathbb C})_{(i)}]^{u_{\zeta,\Gamma}({\mathfrak t})}\Rightarrow \operatorname{H}^{i+j}(u_{\zeta}({\mathfrak b}),{\mathbb C}). 
\end{equation} 
We can now verify the requisite finite generation assumptions on the cohomology for $u_{\zeta,\Gamma}({\mathfrak b})$. 

\begin{theorem}\label{T:bcohofg}  Let ${\ell}$ satisfy Assumption~\ref{A:assumption1}, $\zeta$ be an ${\ell}$th root of unity, and 
$u_{\zeta,\Gamma}({\mathfrak b})$ be a small quantum group for a Borel subalgebra. Then 
\begin{itemize} 
\item[(a)] $\operatorname{H}^{\bullet}(u_{\zeta,\Gamma}({\mathfrak b}),{\mathbb C})$ is a finitely generated ${\mathbb C}$-algebra;
\item[(b)] For any finite-dimensional $u_{\zeta,\Gamma}({\mathfrak b})$-module, $M$, $\operatorname{H}^{\bullet}(u_{\zeta,\Gamma}({\mathfrak b}),M)$ is 
finitely generated over $\operatorname{H}^{\bullet}(u_{\zeta,\Gamma}({\mathfrak b}),{\mathbb C})$. 
\end{itemize} 
\end{theorem} 

\begin{proof} (a) Let $R:=\operatorname{H}^{\bullet}(u_{\zeta,\Gamma}({\mathfrak b}),{\mathbb C})$. 
From Theorem~\ref{T:ucohofg}(a), and the spectral sequence (\ref{eq:spectral2}), we have polynomial ring 
$S:=S^{\bullet}({\mathfrak u}^{*})^{[1]}$ with $d_{r}(S)=0$ for $r\geq 1$. Consequently, $R$ finitely generated over $S$. This shows (a). 

(b) By using induction on the composition length of $M$ and the long exact sequence in cohomology one can reduce the statement to showing that 
 $\operatorname{H}^{\bullet}(u_{\zeta,\Gamma}({\mathfrak b}),M)$ is finitely generated over $R$ for $M$ a simple $u_{\zeta,\Gamma}({\mathfrak b})$-module. 

The simple $u_{\zeta,\Gamma}({\mathfrak b})$-modules are one-dimensional and indexed by $\lambda\in \widehat{\Gamma/\Gamma'}$. 
By using the LHS spectral sequence, one has  
$$\operatorname{H}^{\bullet}(u_{\zeta,\Gamma}({\mathfrak b}),\lambda)\cong 
\operatorname{Hom}_{u_{\zeta,\Gamma}({\mathfrak t})}(-\lambda,\operatorname{H}^{\bullet}(u_{\zeta}({\mathfrak u}),{\mathbb C}))=A_{\lambda}.$$ 
Now $S$ acts on $\operatorname{H}^{\bullet}(u_{\zeta,\Gamma}({\mathfrak b}),\lambda)$ and thus acts on $A_{\lambda}$. This 
action is compatible with the action on $T=\operatorname{H}^{\bullet}(u_{\zeta}({\mathfrak u}),{\mathbb C})$. We have $T\cong \oplus_{\lambda\in \widehat{\Gamma/\Gamma'}} A_{\lambda}$, 
and by Theorem~\ref{T:ucohofg}, $T$ is finitely generated over $S$. Consequently, $A_{\lambda}$ is finitely generated over $S$, thus finitely generated 
over $R$. 
\end{proof} 

\subsection{Calculation of the Cohomology Ring} In this section we calculate the cohomology ring 
$R:=\operatorname{H}^{\bullet}(u_{\zeta,\Gamma}({\mathfrak b}),{\mathbb C})$ for ${\ell}>h$. 
We will need the following fact proved by Andersen and Jantzen \cite[\S 2.2 statement (2)]{AJ}. 

\begin{lemma}\label{L:fixedpoints} \cite{AJ} Let $\rootsys$ be an irreducible root system. For every weight $\lambda$ of $\Lambda^\bullet({\mathfrak u}^*)$ 
and simple root $\alpha_i$, 
\[
|\langle \lambda, \alpha_i\spcheck \rangle + 1 | \leq h -1,
\]
where $h$ is the Coxeter number for $\rootsys$. 
\end{lemma} 
The following theorem provides a natural generalization to the fundamental result of 
Ginzburg and Kumar \cite[Theorem 2.5]{GK}. 

\begin{theorem} 
\label{cohom}
Let ${\ell}$ satisfy Assumption~\ref{A:assumption2} (in particular, ${\ell}>h$), $\zeta$ be an ${\ell}$th root of unity, and 
$u_{\zeta,\Gamma}({\mathfrak b})$ be a small quantum group for a Borel subalgebra. Then 
\begin{itemize}
\item[(a)] $\operatorname{H}^{2\bullet}(u_{\zeta,\Gamma}({\mathfrak b}),{\mathbb C})\cong S^{\bullet}({\mathfrak u}^{*})^{[1]}$; 
\item[(b)] $\operatorname{H}^{2\bullet+1}(u_{\zeta,\Gamma}({\mathfrak b}),{\mathbb C})=0$. 
\end{itemize} 
\end{theorem} 

\begin{proof} Consider the spectral sequence (\ref{eq:spectral2}) and 
$$\operatorname{H}^{n}(\text{gr }u_{\zeta}({\mathfrak u}),{\mathbb C})^{u_{\zeta,\Gamma}({\mathfrak t})}
\cong \bigoplus_{2a+b=n} S^{a}({\mathfrak u}^{*})^{[1]}\otimes [\Lambda^{b}_{\zeta}]^{u_{\zeta,\Gamma}({\mathfrak t})}.$$ 
The ${u_{\zeta,\Gamma}({\mathfrak t})}$-weights of $\Lambda^{b}_{\zeta}$ come from the ${\mathfrak t}$-weights of $\Lambda^\bullet({\mathfrak u}^*)$. 
If $\lambda$ is a weight of $\Lambda^\bullet({\mathfrak u}^*)$ corresponding to an element in $[\Lambda^{b}_{\zeta}]^{u_{\zeta,\Gamma}({\mathfrak t})}$, then 
$\langle \lambda, \Gamma \rangle \subseteq \ell {\mathbb Z}$. Therefore, $\langle \lambda, \alpha_i \rangle \in \ell {\mathbb Z}$ for all $1 \leq i \leq n$. For each simple root $\alpha_i$ of $\rootsys$ we have
\[
\langle \lambda, \alpha_i\spcheck \rangle = \frac{1}{d_i} \langle \lambda, \alpha_i \rangle.
\]
Since $\langle \lambda, \alpha_i\spcheck \rangle$ is an integer, $\langle \lambda, \alpha_i \rangle \in \ell {\mathbb Z}$ and $\gcd(\ell, d_i) =1$, 
we have that that $\langle \lambda, \alpha_i\spcheck \rangle$ is a multiple of $\ell$. Lemma \ref{L:fixedpoints} gives that 
\[
|\langle \lambda, \alpha_i\spcheck \rangle| \leq h < \ell. 
\]
The combination of the two facts implies that $\langle \lambda, \alpha_i\spcheck \rangle =0$ for all simple roots $\alpha_i$. Thus $\lambda =0$ and 
\begin{equation} \label{eq:exteriorfixedpts} 
[\Lambda^{b}_{\zeta}]^{u_{\zeta,\Gamma}({\mathfrak t})}\cong 
\begin{cases} 0 & \text{if $b>0$} \\
 {\mathbb C} & \text{if $b=0$}. 
\end{cases} 
\end{equation} 

Consequently, the $E_{1}^{i,j}$-term of the spectral sequence only contains terms of the form $S^{a}({\mathfrak u}^{*})^{[1]}$ where $2a=i+j$. 
From Theorem~\ref{T:bcohofg}, $d_{r}(S^{\bullet}({\mathfrak u}^{*})^{[1]})=0$ for $r\geq 1$. Thus, the spectral sequence (\ref{eq:spectral2}) 
collapses and yields (a) and (b). 
\end{proof}

\subsection{Classification of Tensor Ideals} Let $\stmod(u_{\zeta,\Gamma}({\mathfrak b}))$ be the 
stable module category of finitely generated $u_{\zeta,\Gamma}({\mathfrak b})$-modules. 
The stable module category for all $u_{\zeta,\Gamma}({\mathfrak b})$-modules will be denoted by 
$\StMod(u_{\zeta,\Gamma}({\mathfrak b}))$. The category $\stmod(u_{\zeta,\Gamma}({\mathfrak b}))$ is 
a monoidal triangulated category. The goal of this section will be to describe the thick tensor ideals in 
$\stmod(u_{\zeta,\Gamma}({\mathfrak b}))$ and its Balmer spectrum. 

Let $R:=\operatorname{H}^{\bullet}(u_{\zeta,\Gamma}({\mathfrak b}),{\mathbb C})$ be the cohomology ring 
for the small quantum group $u_{\zeta,\Gamma}({\mathfrak b})$. In Theorem~\ref{T:bcohofg}(a), it was shown that 
$R$ is a a finitely generated ${\mathbb C}$-algebra. Therefore, $Y=\text{Proj}(R)$, 
the space of (nontrivial) homogeneous prime ideals of $R$, is a Noetherian topological space. In fact, $Y$ is a Zariski space. 

For brevity, the set of subsets, closed subsets, and specialization-closed 
subsets of $Y$ will be denoted respectively by $\mc{X}, \mc{X}_{cl},$ and $\mc{X}_{sp}$.
The finite generation result in Theorem~\ref{T:bcohofg}(b) can be used to define a (cohomological) support variety theory for $u_{\zeta,\Gamma}({\mathfrak b})$. 
Let $W(-)$ be the cohomological support $\stmod ( u_{\zeta, \Gamma}(\mathfrak{b})) \to \mc{X}_{cl}$, defined by
$$W(M)=\{ \mathfrak{p} \in \Proj R: \Ext^\bullet(M, M)_{\mf{p}} \not = 0 \}.$$ This extends to a support map $\StMod( u_{\zeta, \Gamma}(\mathfrak{b})) \to \mc{X}_{sp}$ by \cite[Theorem 5.5]{BIK1}, which we will also denote by $W(-)$. 

Let 
\[
\Phi= \Phi_{W} : \{\text{thick right ideals of $\stmod(u_{\zeta,\Gamma}({\mathfrak b}))$}\}  \to \mc{X}
\]
be the map given by \eqref{Phi}. Note that it takes values in $\mc{X}_{sp}$
because $W(M) \in \mc{X}_{cl}$ for all $M \in \stmod(u_{\zeta,\Gamma}({\mathfrak b}))$. 
On the other hand, we can define an assignment 
\[
\Theta : \mc{X}_{sp} \to \{\text{thick right ideals of $\stmod(u_{\zeta,\Gamma}({\mathfrak b}))$}\} 
\]
by 
$$
\Theta(Z)=\{M\in \stmod(u_{\zeta,\Gamma}({\mathfrak b})) \mid W(M)\subseteq Z\} \quad 
\mbox{for} \quad Z\in \mc{X}_{sp}.
$$ 
We can now state the theorem that classifies thick ideals in $\stmod(u_{\zeta,\Gamma}({\mathfrak b}))$. 
Our results extend the results due to the authors in \cite[Theorems 8.2.1, 8.3.1]{NVY1}. 

\begin{theorem} \label{T:classificationl} Let $u_{\zeta,\Gamma}({\mathfrak b})$ be the small quantum group for the Borel subalgebra 
for an arbitrary finite dimensional complex simple Lie algebra. Assume that ${\ell}$ satisfies Assumption~\ref{A:assumption2} (in particular, ${\ell}>h$), 
which implies that $R\cong S^{\bullet}({\mathfrak u}^{*})$. 

\begin{itemize}
\item[(a)] The above $\Phi$ and $\Theta$ are mutually inverse bijections
$$
\{\text{thick right ideals of $\stmod(u_{\zeta,\Gamma}({\mathfrak b}))$}\} \begin{array}{c} {\Phi} \atop {\longrightarrow} \\ {\longleftarrow}\atop{\Theta} \end{array}  
\{\text{specialization closed sets of $\operatorname{Proj}(R)$} \}. 
$$
\item[(b)] Every thick right ideal of $\stmod(u_{\zeta,\Gamma}({\mathfrak b}))$ is two-sided. 
\item[(c)] There exists a homeomorphism $f:$ $\operatorname{Proj}(R) \to \Spc(\stmod(u_{\zeta,\Gamma}({\mathfrak b})))$. 
\end{itemize} 
\end{theorem}
For the proof of the theorem we will need the following auxiliary lemma
\begin{lemma}
\label{Pi-supp} In the setting of Theorem \ref{T:classificationl}, for every finite dimensional $u_{\zeta,\Gamma}({\mathfrak b})$-module $Q$ and its dual $Q^*$,
\[
W(Q) = W(Q^*).
\]
\end{lemma}
\begin{proof} Every object of $\stmod(u_{\zeta,\Gamma}({\mathfrak b}))$ is rigid. The first composition in \eqref{eq:rigid}
gives that if $Q$ is a finite dimensional $u_{\zeta,\Gamma}({\mathfrak b})$-module, 
then $Q$ is a summand of $Q \otimes Q^* \otimes Q$. So, 
\[
W(Q) \subseteq W(Q \otimes Q^* \otimes Q).
\]
Since $Q$ has a composition series by subquotients isomorphic to the one dimensional modules 
$\lambda\in \widehat{\Gamma/\Gamma'}$,
\[
W(Q \otimes Q^* \otimes Q) = \bigcup_{\lambda\in \widehat{\Gamma/\Gamma'}} W(\lambda \otimes Q^* \otimes Q).
\]
The cohomological support $W$ is automatically a quasi support datum. Applying this fact and Proposition \ref{P:properties}(b-c), we obtain that
\[
W(\lambda \otimes Q^* \otimes Q) \subseteq W((Q^*)^{\gamma_\lambda} \otimes \lambda \otimes Q) \subseteq W((Q^*)^{\gamma_\lambda}) = W(Q^*)
\]
for all $\lambda\in \widehat{\Gamma/\Gamma'}$. Combining the above inclusions gives $W(Q) \subseteq W(Q^*)$. 
Since the square of the antipode of $u_{\zeta,\Gamma}({\mathfrak b})$ is an inner automorphism, $Q^{**} \cong Q$. 
Interchanging the roles of $Q$ and $Q^*$ gives $W(Q^*) \subseteq W(Q)$. Hence, $W(Q) = W(Q^*)$. 
\end{proof}
\begin{proof}[Proof of Theorem \ref{T:classificationl}]
(a) This statement follows by \cite[Theorem 7.4.3]{NVY1}. The (fg) assumption is established in Theorem~\ref{T:bcohofg}.
The arguments in \cite[Section 7.4]{BKN}, Lemma \ref{Pi-supp}, and the faithfulness of the cohomological support verify \cite[Assumption 7.2.1]{NVY1}. 

We will prove (b) and (c) by an analogous argument to \cite[Theorem 6.2.1]{NVY1}. As noted earlier, the cohomological support $W$ is a quasi support datum and satisfies \cite[Assumption 7.2.1]{NVY1}. 
One now needs to show that $W$ satisfies the following 
property:

(Realization) If $V$ is a closed set in $Y$, then there exists a compact object $M$ with $\Phi( \langle M \rangle) =V$.

We compute:
\begin{align*}
\Phi ( \langle M \rangle) &= \bigcup_{C, D \in \bK^c} W( C \otimes M \otimes D)\\
&= \bigcup_{C \in \bK^c} W(C \otimes M)\\
&= \bigcup_{\lambda \in \widehat{\Gamma/\Gamma'}} W(\lambda \otimes M)\\
&= \bigcup_{\lambda \in \widehat{\Gamma/\Gamma'}} W(\lambda \otimes M \otimes \lambda^{-1})\\
&= \bigcup_{\lambda \in \widehat{\Gamma/\Gamma'}} W(M^{\gamma_\lambda})\\
&= \Pi\cdot W(M)\\
&= W(M).
\end{align*}
The second and fourth equalities follow from the fact that $W$ is a quasi support datum, the fourth since
$$W(\lambda \otimes M) \subseteq W(\lambda \otimes M \otimes \lambda^{-1} ) \subseteq W(\lambda \otimes M \otimes \lambda^{-1} \otimes \lambda) = W(\lambda \otimes M).$$ The third, fifth, and seventh equalities follow from 
Proposition~\ref{P:properties}, parts (a), (b), and (d) respectively. Since $\Phi(\langle M \rangle ) = W(M)$ and every closed set of $\Proj R$ may be realized as $W(M)$ for some compact $M$, $W$ satisfies the Realization Property. 

Analogously to the proof of \cite[Theorem 6.2.1]{NVY1}, the conditions that $W$ is a quasi support satisfying \cite[Assumption 7.2.1]{NVY1} and the Realization Property allow one to conclude that there exists an order-preserving bijection:
$$
\{\text{thick two-sided ideals of $\stmod(u_{\zeta,\Gamma}({\mathfrak b}))$}\} \begin{array}{c} {\Phi} \atop {\longrightarrow} \\ {\longleftarrow}\atop{\Theta} \end{array}  
\{\text{specialization closed sets of $\operatorname{Proj}(R)$} \}. 
$$
Since we already know by (a) that $\Phi$ induces a bijection between the thick right ideals of $\stmod(u_{\zeta,\Gamma}({\mathfrak b}))$ and specialization closed sets of $\operatorname{Proj}(R)$, it follows immediately that every thick right ideal is two-sided. 

In order to obtain part (c), we must show that $\Phi$ is a weak support datum. Let $\bI$ and $\bJ$ be two thick ideals of $\stmod(u_{\zeta,\Gamma}({\mathfrak b}))$. We claim that $\langle \bI \otimes \bJ \rangle = \bI \cap \bJ$. It is clear that $\langle \bI \otimes \bJ \rangle \subseteq \bI \cap \bJ$, by definition. Both thick ideals $\langle \bI \otimes \bJ \rangle$ and $\bI \cap \bJ$ of $\stmod(u_{\zeta,\Gamma}({\mathfrak b}))$ are semiprime, by Proposition \ref{rigid-semiprime}. In other words, 
\begin{align*}
\langle \bI \otimes \bJ \rangle = &\bigcap \{ \bP \in \Spc(\stmod(u_{\zeta,\Gamma}({\mathfrak b}))) : \langle \bI \otimes \bJ \rangle \subseteq \bP\}\\
= &\bigcap \{ \bP \in \Spc(\stmod(u_{\zeta,\Gamma}({\mathfrak b}))) : \bI \subseteq \bP\} \cap \\
&\bigcap  \{ \bP \in \Spc(\stmod(u_{\zeta,\Gamma}({\mathfrak b}))) : \bJ \subseteq \bP\},\\
\end{align*}
and
\begin{align*}
\bI \cap \bJ & = \bigcap \{ \bP \in \Spc(\stmod(u_{\zeta,\Gamma}({\mathfrak b}))) : \bI \cap \bJ \subseteq \bP\}.\\
\end{align*}
Then it is clear that $\bI \cap \bJ \subseteq \langle \bI \otimes \bJ \rangle$, since each prime ideal containing either $\bI$ or $\bJ$ must necessarily contain $\bI \cap \bJ$. Therefore, $\bI \cap \bJ = \langle \bI \otimes \bJ \rangle$. By (a), $\Phi$ gives an order-preserving bijection between thick two-sided ideals of $\stmod(u_{\zeta,\Gamma}({\mathfrak b}))$ and specialization closed sets of $\operatorname{Proj}(R)$, which shows that
\begin{align*}
\Phi(\langle \bI \otimes \bJ \rangle) &= \Phi(\bI \cap \bJ)\\
&= \Phi(\bI) \cap \Phi(\bJ).
\end{align*}
Therefore $W$ is a weak support datum, and \cite[Theorem 6.2.1]{NVY1} gives part (c). 
\end{proof}

\subsection{The Tensor Product Property for the Cohomological Support Map}\label{q-Borels}
In this section we illustrate Theorem \ref{compl-prime}. We prove that the cohomological support maps for all small quantum Borel algebras
associated to arbitrary complex simple Lie algebras and arbitrary choices of group-like elements have the tensor product property. 
This was conjectured by Negron and Pevtsova \cite{NP} and proved by them in the type $A$ case.
\begin{theorem}
\label{quant-Borel-spc}  
Let $u_{\zeta,\Gamma}({\mathfrak b})$ be the small quantum group for the Borel subalgebra 
of an arbitrary finite dimensional complex simple Lie algebra and a lattice ${\mathbb Z} \rootsys \subseteq \Gamma \subseteq X$.
Assume that ${\ell}$ satisfies Assumption~\ref{A:assumption2} 
(in particular, ${\ell}>h$). Then the following hold:
\begin{enumerate}
\item[(a)] All prime ideals of $\stmod(u_{\zeta,\Gamma}({\mathfrak b}))$ are completely prime.
\item[(b)] The cohomological support
\[
W (-): \stmod(u_{\zeta,\Gamma}({\mathfrak b})) \to \mc{X}_{cl} (\operatorname{Proj}(\operatorname{H}^{\bullet}(u_{\zeta,\Gamma}({\mathfrak b}),{\mathbb C})))
\]
has the tensor product property $W(A \otimes B)  = W(A) \cap W(B)$ for all $A, B \in \stmod(u_{\zeta, \Gamma}({\mathfrak b}))$.
\end{enumerate}
\end{theorem}
\begin{proof} Part (a) of the theorem follows by combining Theorems \ref{compl-prime} and \ref{quant-Borel-spc}(a). 

(b) Recall the universal support datum 
\[
V : \stmod(u_{\zeta,\Gamma}({\mathfrak b}))\to  \mc{X}_{cp} ( \Spc (\stmod(u_{\zeta,\Gamma}({\mathfrak b}))))
\]
defined in Section~\ref{support}. It follows from Theorem \ref{equiv} and part (a) of this theorem that $V$ has the tensor product property. 

In the proof of Theorem \ref{T:classificationl} it was shown that $W$ is a weak support datum. 
By Theorem \ref{T:universality}(b), there exists a homeomorphism 
\[
f:  \operatorname{Proj}(\operatorname{H}^{\bullet}(u_{\zeta,\Gamma}({\mathfrak b}),{\mathbb C})) \to \Spc (\stmod(u_{\zeta,\Gamma}({\mathfrak b})))
\]
satisfying $\Phi_{W}(\langle M \rangle)= f^{-1}(V(M))$ for all $M \in \stmod(u_{\zeta}({\mathfrak b}))$. 
Applying Theorem \ref{T:classificationl}(b), \eqref{eq:ideals} and the fact that $W$ is a quasi support datum, we obtain 
\[
W(M) \subseteq \Phi(\langle M \rangle) = \Phi(\langle M \rangle_{\mathrm{r}}) \subseteq W(M)
\]
for all $M \in \stmod(u_{\zeta}({\mathfrak b}))$. Therefore, 
\[
W(M) = \Phi(\langle M \rangle)= f^{-1}(V(M)), \quad \forall M \in \stmod(u_{\zeta}({\mathfrak b})).
\]
Now Theorem \ref{compl-prime}, the continuity of $f$ and the fact that the universal support datum $V$ has the tensor product property give
\begin{align*}
W(A \otimes B) &= f^{-1}(V(A \otimes B)) = f^{-1}(V(A) \cap V(B)) 
\\
&= f^{-1}(V(A)) \cap   f^{-1}(V(B)) =  W(A) \cap W(B)
\end{align*}
for all $A, B \in \stmod(u_{\zeta}({\mathfrak b}))$. 
\end{proof} 
Example \ref{ex-X'} and Theorem \ref{quant-Borel-spc} imply the following:
\begin{corollary}
Let $u_{\zeta}({\mathfrak b})$ be the standard small quantum group for the Borel subalgebra 
of an arbitrary finite dimensional complex simple Lie algebra. 
Assume that ${\ell}$ satisfies Assumption~\ref{A:assumption2} and that $\ell$ is coprime to $|X/{\mathbb Z} \rootsys|$.
Then the following hold:
\begin{enumerate}
\item[(a)] All prime ideals of $\stmod(u_{\zeta}({\mathfrak b}))$ are completely prime.
\item[(b)] The cohomological support
\[
W (-): \stmod(u_{\zeta}({\mathfrak b})) \to \mc{X}_{cl} (\operatorname{Proj}(\operatorname{H}^{\bullet}(u_{\zeta}({\mathfrak b}),{\mathbb C})))
\]
has the tensor product property $W(A \otimes B)  = W(A) \cap W(B)$ for all $A, B \in \stmod(u_{\zeta}({\mathfrak b}))$.
\end{enumerate}
\end{corollary} 
\bre{gamma-ell-borel} Assume that ${\ell}$ satisfies Assumption~\ref{A:assumption2} and that $\ell$ is coprime to $|X/{\mathbb Z} \rootsys|$.
Then by Proposition \ref{P:G'}(b), the small quantum Borel subalgebra $u_{\zeta,\Gamma}({\mathfrak b})$ is based off the group algebra of the lattice $\Gamma/\ell \Gamma$, 
cf. \eqref{group-alg}.
Therefore, the statements in parts (a) and (b) of Theorem \ref{quant-Borel-spc} hold for the version of a small quantum Borel subalgebra
based off the group algebra of the lattice $\Gamma/\ell \Gamma$. 
\ere

\subsection{The Negron--Pevtsova small quantum Borel algebras}
\label{Borels-NP}
In \cite{N,NP} Negron and Pevtsova considered a different version of small quantum Borel subalgebras. 
For a lattice, $\Gamma$, with ${\mathbb Z}\rootsys \subseteq \Gamma \subseteq X$, set
\[
\Gamma^\perp:= \{ \nu \in \Gamma \mid \langle \nu, \Gamma \rangle \subseteq \ell {\mathbb Z} \}. 
\]
Denote the canonical projection
\[
\Gamma \twoheadrightarrow \Gamma/ \Gamma^\perp \quad \mbox{by} \quad \mu \mapsto \overline{\overline{\mu}}.
\]
Let
\[
\widetilde{u}_{\zeta,\Gamma}({\mathfrak t}) \; \; \mbox{denote the group algebra of $\Gamma/\Gamma^\perp$ over ${\mathbb C}$}. 
\]
For $\mu \in \Gamma/\Gamma^\perp$ denote by $K_\mu$ the corresponding element of $u_{\zeta,\Gamma}({\mathfrak t})$.  
Following \cite{N,NP}, define the Hopf algebra
\[
\widetilde{u}_{\zeta,\Gamma}({\mathfrak b})= u_{\zeta}({\mathfrak u}) \# \widetilde{u}_{\zeta,\Gamma}({\mathfrak t})
\]
with relations 
\[
K_{\mu} E_{\alpha} K_{\mu}^{-1}=\zeta^{\langle \alpha,\mu_0 \rangle} E_{\alpha} \quad \mbox{for}
\quad \mu \in \Gamma/ \Gamma', \alpha \in \rootsys^+,
\]
where $\mu_0 \in \Gamma$ is a preimage of $\mu$. By the definition of the lattice $\Gamma^\perp$, the right hand side does not  
depend on the choice of preimage. The coproduct of the generators $E_{\alpha_i}$ is given by
\begin{equation}
\label{coprod2}
\Delta(E_{\alpha_i}) = E_{\alpha_i} \otimes 1 + K_{\overline{\overline{\alpha_i}}} \otimes E_{\alpha_i}
\end{equation}
for $1 \leq i \leq n$. The antipode is given by $S( E_{\alpha_i}) = - K_{\overline{\overline{\alpha_i}}}^{-1} E_{\alpha_i}$.

Clearly, $\Gamma' \supseteq \Gamma^\perp$ and the elements
\[
\{ K_{\mu} \mid \mu \in \Gamma'/ \Gamma^\perp \} 
\]
are in the center of $\widetilde{u}_{\zeta,\Gamma}({\mathfrak b})$. In other words,  
$\widetilde{u}_{\zeta,\Gamma}({\mathfrak b})$ has a larger center than $u_{\zeta,\Gamma}({\mathfrak b})$.

By abuse of notation we will denote by $\mu \mapsto \overline{\mu}$ the canonical projection $\Gamma/\Gamma^\perp \twoheadrightarrow \Gamma/ \Gamma'$, 
recall \eqref{bar}. There exists a surjective Hopf algebra homomorphism
\[
\widetilde{u}_{\zeta,\Gamma}({\mathfrak t}) \twoheadrightarrow u_{\zeta,\Gamma}({\mathfrak t}) 
\]
given by $K_\mu \mapsto K_{\overline{\mu}}$ for $\mu \in \Gamma/\Gamma^\perp$ and 
$E_\alpha \mapsto E_\alpha$ for $\alpha \in \rootsys^+$. Its kernel is the ideal generated by the central elements
\[
\{ K_{\mu} - 1 \mid \mu \in \Gamma'/ \Gamma^\perp \}. 
\]

Let $d$ be the minimal positive integer such that the restriction of $\langle -, - \rangle$ to $\Gamma$ takes values in $ {\mathbb Z}/d$. Choose 
a primitive $(d\ell)$th root of unity $\xi$ such that $\zeta = \xi^d$. Consider the 
symmetric (multiplicative) bicharacter 
\[
\chi : \Gamma/\Gamma^\perp \times \Gamma/\Gamma^\perp \to {\mathbb C}^\times \quad \mbox{given by} \quad
\chi(\mu, \nu) := \xi^{ \langle \mu, \nu \rangle} \; \; \mbox{for $\mu, \nu \in \Gamma/\Gamma^\perp$},
\]  
where $\mu_0$ and $\nu_0$ are preimages of $\mu$ and $\nu$ in $\Gamma$. By the definition of $\Gamma^\perp$, the bicharacter
is well-defined and nondegenerate. It induces the isomorphism
\begin{equation}
\label{phi}
\varphi : \Gamma/\Gamma^\perp \stackrel{\cong}{\longrightarrow} \widehat{\Gamma/\Gamma^\perp} \quad \mbox{given by} \quad
\varphi(\mu) := \chi(\mu, -)  \; \; \mbox{for $\mu \in \Gamma/\Gamma^\perp$}.
\end{equation}
Similarly to the discussion for $u_{\zeta,\Gamma}({\mathfrak t})$, for $\lambda \in \widehat{\Gamma/\Gamma^\perp}$
define the one dimensional representation of $\widetilde{u}_{\zeta,\Gamma}({\mathfrak t})$
\[
K_\mu \mapsto \lambda(\mu), \quad
E_\alpha \mapsto 0, \quad \forall \mu \in \Gamma/\Gamma^\perp, \alpha \in \rootsys^+.
\]
The irreducible representations of 
$\widetilde{u}_{\zeta,\Gamma}({\mathfrak t})$ are one-dimensional and are indexed by $\widehat{\Gamma/\Gamma^\perp}$. 
We have a much simplified version of Proposition \ref{P:properties} for the algebras $\widetilde{u}_{\zeta,\Gamma}({\mathfrak t})$:
\begin{proposition}  
\label{N-properties} \cite{N}
\begin{itemize}
\item[(a)] The irreducible representations for $\widetilde{u}_{\zeta,\Gamma}({\mathfrak b})$ are one-dimensional and are precisely the 
representations $\lambda$ for $\lambda \in \widehat{\Gamma/\Gamma^\perp}$.
\item[(b)] For any $\widetilde{u}_{\zeta,\Gamma}({\mathfrak b})$-module, $Q$, and $\lambda \in \widehat{\Gamma/\Gamma^\perp}$ one has 
\[ 
\lambda \otimes Q \otimes \lambda^{-1} \cong Q.
\]
\end{itemize} 
\end{proposition} 
Part (a) is proved in the same way as Proposition \ref{P:properties}(a). 
Part (b) follows at once by combining the following two facts:
\begin{enumerate}
\item  For any $\widetilde{u}_{\zeta,\Gamma}({\mathfrak b})$-module, $Q$, and $\lambda \in \widehat{\Gamma/\Gamma^\perp}$, $\lambda \otimes Q \otimes \lambda^{-1} \cong Q^{\gamma''_\lambda}$ 
where, ${\gamma''_\lambda}$ is the automorphism of $\widetilde{u}_{\zeta,\Gamma}({\mathfrak b})$
given by 
\[
\gamma''_\lambda (E_\alpha)=  \lambda(\overline{\overline{\alpha}}) E_\alpha, \quad \gamma_\lambda(K_\mu)=K_\mu, 
\quad \forall \mu \in \Gamma/\Gamma^\perp, \alpha \in \rootsys^+
\]
(this follows from \eqref{coprod2}); 
\item $\gamma''_\lambda$ equals the an inner automorphism $x \mapsto K_{\varphi^{-1}(\mu)} x K_{\varphi^{-1}(\mu)}^{-1}$ (this follows from \eqref{phi}).
\end{enumerate} 

From this point further the proofs of Theorems \ref{cohom}, \ref{quant-Borel-spc}, and \ref{quant-Borel-spc}, extend mutatis mutandis from the family of algebras
$u_{\zeta,\Gamma}({\mathfrak b})$ to the family of algebras $\widetilde{u}_{\zeta,\Gamma}({\mathfrak b})$.  Furthermore, there is a simplification in 
the proof of the analog of Theorem \ref{quant-Borel-spc}: on the third line of the long display $\lambda \otimes M \otimes \lambda^{-1} \cong M$
and the rest of the equalities in the display can be omitted. This proves the following:
\begin{theorem}
\label{NG-version}  
Let $\widetilde{u}_{\zeta,\Gamma}({\mathfrak b})$ be the version of the small quantum group for the Borel subalgebra 
of an arbitrary finite dimensional complex simple Lie algebra and a lattice ${\mathbb Z} \rootsys \subseteq \Gamma \subseteq X$ defined in \cite{N}. 
Assume that ${\ell}$ satisfies Assumption~\ref{A:assumption2}. Then the following hold:
\begin{enumerate}
\item[(a)] $\operatorname{H}^{2\bullet+1}(\widetilde{u}_{\zeta,\Gamma}({\mathfrak b}),{\mathbb C})=0$ 
and $R:=\operatorname{H}^{2\bullet}(\widetilde{u}_{\zeta,\Gamma}({\mathfrak b}),{\mathbb C})\cong S^{\bullet}({\mathfrak u}^{*})^{[1]}$. 
\item[(b)] There exist two mutually inverse bijections
$$
\{\text{thick right ideals of $\stmod(\widetilde{u}_{\zeta,\Gamma}({\mathfrak b}))$}\} \begin{array}{c} {\Phi} \atop {\longrightarrow} \\ {\longleftarrow}\atop{\Theta} \end{array}  
\{\text{specialization closed sets of $\operatorname{Proj}(R)$} \},
$$
where $\Phi$ and $\Theta$ are given by
$$
\Phi(\bI) := \bigcup_{A \in \bI} W(A)
$$
for the cohomological support $W : \stmod(\widetilde{u}_{\zeta,\Gamma}({\mathfrak b})) \to \mc{X}_{cl}(\operatorname{Proj}(R))$ and
$$
\Theta(Z):=\{M\in \stmod(u_{\zeta,\Gamma}({\mathfrak b})) \mid W(M)\subseteq Z\} \quad 
\mbox{for} \quad Z\in \mc{X}_{sp}(\operatorname{Proj}(R)).
$$ 
\item[(c)] Every thick right ideal of $\stmod(\widetilde{u}_{\zeta,\Gamma}({\mathfrak b}))$ is two-sided. 
\item[(d)] There exists a homeomorphism $\operatorname{Proj}(R) \cong \Spc(\stmod(u_{\zeta,\Gamma}({\mathfrak b})))$. 
\item[(e)] All prime ideals of $\stmod(\widetilde{u}_{\zeta,\Gamma}({\mathfrak b}))$ are completely prime.
\item[(f)] The cohomological support
\[
W (-): \stmod(\widetilde{u}_{\zeta,\Gamma}({\mathfrak b})) \to \mc{X}_{cl} (\operatorname{Proj} R)
\]
has the tensor product property $W(A \otimes B)  = W(A) \cap W(B)$ for all $A, B \in \stmod(\widetilde{u}_{\zeta, \Gamma}({\mathfrak b}))$.
\end{enumerate}
\end{theorem}
There is a further simplification in the proof of part (c) of the theorem compared to that of Theorem \ref{T:classificationl}(b). 
Since the algebras $\widetilde{u}_{\zeta,\Gamma}({\mathfrak b})$ satisfy the property in Proposition \ref{N-properties}(b), 
part (c) of the theorem also follows directly from this property.

\end{document}